\documentclass[11pt,letterpaper]{amsart}
\usepackage{srcltx}

\usepackage{amssymb}
\usepackage{amsmath}
\usepackage{amsthm}
\usepackage{amscd}
\usepackage{graphicx}
\usepackage{amsfonts}
\newtheorem{thm}{Theorem}[section]

\newtheorem{coro}[thm]{Corollary}
\newtheorem{lemm}[thm]{Lemma}
\newtheorem{prop}[thm]{Proposition}
\newtheorem{prob}[thm]{Problem}

\theoremstyle{definition}
\newtheorem{defi}[thm]{Definition}

\theoremstyle{remark}
\newtheorem{rema}[thm]{Remark}

\newenvironment{blue}{\color{blue}}{}

\date{\textsf{}}

\newcommand{\dist}[2]{{\rm dist\,}}

\begin{document}

\title{Algorithmically finite groups}


\author{Alexei Miasnikov}
\address{Department of Mathematical Sciences\\Stevens Institute of Technology\\Hoboken, NJ 07030}
\email{amiasnikov@gmail.com}

\author{Denis Osin}
\address{Mathematics Department\\
Vanderbilt University\\
Nashville, TN 37240}
\email{denis.v.osin@vanderbilt.edu}

\thanks{The research of the first author was supported by
NSF grant  DMS-0914773.}
\thanks{The research of the second author was supported by
NSF grant  DMS-1006345.}


\maketitle

\begin{abstract}
We call a group $G$ {\it algorithmically finite} if no algorithm can produce an infinite set of pairwise distinct elements of $G$. We construct examples of recursively presented infinite algorithmically finite groups and study their properties. For instance, we show that the Equality Problem is decidable in our groups only on strongly (exponentially) negligible sets of inputs.
\end{abstract}

\bigskip
\section{Introduction}
Recall, that a group $G$ is called {\em recursively presented} if it has a presentation $G = \langle X \mid R\rangle$, where the set of generators $X$ is finite and the set of relators $R$ is a computably enumerable  subset of $F(X)$. Here we view the free group $F(X)$ as the set of freely reduced words in the alphabet $X \cup X^{-1}$ with the standard multiplication (free reduction of the concatenation of two words). A subset $W \subseteq F(X)$ is {\it computably enumerable} if there exists an algorithm $A$ that computes a function $f_A: \mathbb{N} \to F(X)$ with $f_A(\mathbb{N}) = W$. Such a function $f_A$ gives a computable enumeration $W = \{w_1,w_2, \ldots, \}$ of the set $W$, where $w_i = f(i)$.

Recall that the Equality Problem (EP) in a group $G$ generated by a finite set $X$ is, given two words $u,v\in F(X)$, to decide whether $u$ and $v$ represent the same element of $G$. EP for groups is easily reducible to the Word Problem (WP), that is to decide if a given word from $F(X)$ represent the trivial element of $G$. So there is no need to consider them separately, and usually they are both referred to as the Word Problems. However, this may not be the case when one consider decidability of the problems on various subsets of $F(X)$, not the whole group $F(X)$.  In this settings the Equality Problem is the most natural one, and the only one that make sense in semigroups.

We say that EP  is decidable in $G$ on a set of inputs $S \subseteq F(X)$  if there is a partial decision algorithm for EP which halts on all pairs from  $S \times S$. The original Equality Problem for finitely presented groups was formulated by M.Dehn in 1912 \cite{Dehn} (and two years later  by A. Thue for semigroups in a similar fashion \cite{Thue})  in the following way:

\begin{quote}
\it Construct an algorithm to determine for an arbitrary finitely presented group $G = \langle X \mid R\rangle$ and any two words $u$ and $v$ in the alphabet $X \cup X^{-1}$ whether or not $u$ and $v$ represent the same
element of $G$.
\end{quote}

Certainly, Dehn and Thue believed that such an algorithm should exist. Notice, that  not only they asked to find a decision algorithm they were actually asking  to find a  {\em uniform} decision algorithm that would work for all such  groups (and semigroups).

In  1947 Markov \cite{Markov}  and Post \cite{Post}  constructed independently  first  finitely presented semigroups with undecidable EP, and in 1955 Novikov  \cite{Novikov}, and soon after W.W. Boone  \cite{Boone1,Boone2},   constructed  finitely presented groups with undecidable EP. Now there are much shorter  examples of
semigroups with undecidable  word problem constructed by G. S. Tseitin \cite{Tseitin}, D. Scott \cite{Scott}, Matiyasevich \cite{Mat}, Makanin \cite{Makanin}. Other examples of groups with undecidable Word Problem are due to    J. L. Britton \cite{Britton},  V.V. Borisov \cite{Borisov}, and D.J. Collins \cite{Collins}. An excellent exposition of the results in this area  with complete and improved  proofs is given in  the survey \cite{Adian} by S.I. Adian  and V.G. Durnev.

In this paper we introduce and study a class of groups with ``extremely undecidable" EP. We say that a finitely generated group $G$ is {\it algorithmically finite} if there is no algorithmic way to produce an infinite set of pairwise distinct elements of $G$. That is, if $G$ is generated by a finite set $X$, then there exists no computably enumerable subset $S\subseteq F(X)$ such that the natural image of $S$ consists of pairwise distinct elements. In particular, the EP is decidable in algorithmically finite groups only on those subsets $S\subseteq F(X)$, where it is obviously decidable, i.e., which have finite image in $G$. 

Clearly every finite group is algorithmically finite. Our main result shows that the converse does not hold. We say that a group $G$ is a {\em Dehn monster}, if it is recursively presented, infinite, and algorithmically finite.

\begin{thm}\label{main}
Dehn monsters exist. 
\end{thm}

The proof of Theorem \ref{main} is based on new ideas and does not interpret any machines. Instead,  it uses Golod-Shafarevich presentations as a tool to control consequences of relations and analogues of {\em simple} sets from computability theory.  The groups constructed in this way are infinitely presented and the the following problem remains open. 
\begin{prob}\label{prob1}
Does there exist a finitely presented Dehn monster?
\end{prob}
Note that this is a real challenge, since every Dehn monster is an infinite torsion group (see Proposition \ref{properties}) and no examples of finitely presented infinite torsion groups are known.

In the era when much of the focus is on practical computing, the complexity of computations became an important issue.    In \cite{KMSS1,KMSS2}, a new notion of {\em generic complexity} of computations was introduced. In this model one is looking for partial decision algorithms which perform well on typical (generic) sets of inputs. In particular, for a group $G$ with a finite generating set $X$  EP is {\it generically decidable} if there is a partial algorithm $A$ solves the Word Problem in $G$ correctly on most words from  $F(X)$. That is, the halting set of $A$ is {\em generic} with respect to the stratification of $F(X)$ given by the standard length function $|\cdot |$ on $F(X)$. Recall that $T \subseteq F(X)$ is generic in $F(X)$ if
 $$\rho_n(T)  = \frac{|T \cap S_n|}{|S_n|} \to 1 \ for \ n \to \infty, $$ where $S_n = \{w \in F(X) \mid |w| = n\}$. Complements of generic subsets are callled {\em negligible}.

It has been noticed soon that many undecidable problems are generically decidable, and the generic decision algorithms quite often are very efficient. For example, it was shown in \cite{HM} that the famous Halting Problem for Turing machines with one-ended infinite tape is generically decidable in polynomial time, and all the groups and semigroups with undecidable EP mentioned above have linear time generic decision algorithms  \cite{MUW}.

The cryptography quest for algorithmic problems which are hard on most inputs  generated a new waive of interest to ``generically hard" algorithmic problems in group theory (see, for example, the book  \cite{MSUbook}). A new idea on how to ``amplify" the algorithmic hardness of EP in semigroups was introduced in \cite{MR}. It turns out that given a finitely presented semigroup $S$ with undecidable EP one can construct a new finitely presented semigroup $S^\prime$ where EP is undecidable on every generic set of inputs. Unfortunately, this construction does not work for groups. 

This paper was partially motivated by the question of whether groups with ``generically hard" EP exist. Answering this question we prove that infinite algorithmically finite groups satisfy a property which is much stronger than  undecidability of the EP on generic set of inputs. Recall that a subset $S\subseteq F(X)$ is {\em strongly negligible} (or exponentially negligible)  if there exists $t>1$ such that $\rho_n(S) = O(t^{-n})$ as $n\to \infty$. The result below follows immediately from Lemma \ref{changegs} and Theorem \ref{th:Equality}.

\begin{thm}\label{genund}
Let $G$ be an algorithmically finite group generated by a finite set $X$, $S\subseteq F(X)$ a subset with decidable EP in $G$. Then $S$ is negligible. Moreover, if $G$ is non-amenable, then $S$ is strongly negligible.
\end{thm}

In fact, the groups constructed in Theorem \ref{main} are non-amenable. Thus we immediately obtain the following. 

\begin{coro}
There exists a recursively presented group $G$ such that for every finite generating set $X$ of $G$, every subset $S\subseteq F(X)$ with decidable EP in $G$ is strongly negligible.
\end{coro}

Theorem \ref{genund} motivates a weaker version of Problem \ref{prob1}, which is is still open and very intriguing (see \cite{MR,GMO}).
\begin{prob} 
Does there exist a finitely presented group such that a) the EP is decidable only on negligible sets of inputs; b) the EP is undecidable on every generic set of inputs?
\end{prob}
Clearly a) implies b), but the converse is, a priori, not true. The best current result in this direction is Theorem 3 from  \cite{GMO}: {\it if  WP in a finitely presented amenable group $G$  is  undecidable then it is undecidable on every exponentially generic set for any choice of generators  in $G$.} (Recall that a subset $S\subseteq F(X)$ is {\it exponentially generic set} if $F(X)\setminus S$ is strongly negligible.) Such amenable groups do exist \cite{Kh}, their  construction  simulates the work of a Minsky machine  by the defining relations.

\section{Construction}

\subsection{Golod-Shafarevich presentations}

Let us fix a prime $p$ and let $X = \{x_1, . . . , x_d\}$  be a finite set, and
$F = F(X)$ a  free group on $X$. For a fixed prime number $p$ denote by $\Lambda_p$ the  ring   $\mathbb{Z}_p[[u_1, \ldots,u_d]]$ of non-commutative formal power series over the field $\mathbb{Z}_p$ in $d$ variables $u_1, \ldots,u_d$. The
map $X \to \Lambda_p$ given by $x_i \mapsto 1 + u_i$ extends (uniquely) to an injective homomorphism $\pi: F \to \Lambda_p^\ast$, called  the Magnus embedding.

Let $I$ denote the ideal of $\Lambda_p$ generate by $u_1, \ldots , u_d$. The Zassenhaus $p$-series (filtration)
$$F = D_1F > D_2F > \ldots > D_n F > \ldots $$
in $F$ is defined by the dimension subgroups $D_nF$ of $F$, where
$$
D_nF  = \{f \in F \mid f \equiv 1\ mod \  I^n\}
$$
It is not hard to see that for any $f \in F$ there exists a unique $n \geq 1$   (termed the {\em  degree } $deg (g)$ of $f$) such that $f \in D_nF \smallsetminus D_{n+1}F$. Moreover, for any $g, h \in F$
$$deg ([g, h]) = deg (g) + deg (h), \ \ deg (g^p) =  p \cdot deg (g).$$
 It follows that for every $i, j \in \mathbb{N}$
$$
D_1F = F, \ \ [D_iF,D_jF] \subseteq D_{i+j}F, \ \  (D_iF)^p \subseteq D_{ip}F.
$$
In particular, the quotients $F/D_nF$ are finite $p$-groups.

If $G$ is a group generated by $X$ then subgroups of finite index $p^k, k \in \mathbb{N}$, form a basis of the pro-$p$ topology on $G$, which makes $G$ into a topological group. The completion $G_{\hat{p}}$ of $G$ in this topology is the pro-$p$-completion of $G$. If  $G$ is given by a presentation $G = \langle X \mid R\rangle$  then the pro-$p$ completion $G_{\hat{p}}$ of $G$ has the same presentation  $G_{\hat{p}} = \langle X \mid R\rangle$ in the category of pro-$p$ groups. Let  $ d(G_{\hat{p}})$ be  the minimal number of (topological) generators of $G_{\hat{p}}$. Then   $ |X| = d(G_{\hat{p}})$ if and only if $R$ has no relators of degree 1.

Let  $P = \langle X \mid R\rangle$ be a presentation. Denote by $n_i(R)$ the number of relators in $R$ of degree $i$  with respect to
the Zassenhaus $p$-series in $F(X)$.  Consider  the following formal power series:
  $$H_p(X,R,t) = 1-\hat{d}t + \sum\limits_{i = 1}^\infty n_i(R)t^i,$$
  where $\hat{d} = d(G_{\hat{p}})$ is the minimal number of topological generators of the pro-$p$ completion $G_{\hat{p}} $ of the discrete group
$G$ defined by the presentation $\langle X \mid R\rangle$.
It follows from the above that $\hat{d} = |X|$ if  and only if $n_1(R) = 0$.

The presentation $P$ is termed a  {\em Golod-Shafarevich presentation} if there exists $t_0, 0 < t_0 < 1$  such that $H_p(X,R,t_0) < 0$.

The following is the principal result about Golod-Shafarevich presentations. It was proved by   Golod and Shafarevich \cite{GS} with some further improvements by  Vinberg and Gaschutz.

\begin{thm}\label{GSinf}
Let $G$ be an (abstract) group defined by a Golod-Shafarevich presentation. Then the pro-$p$ completion of $G$ is infinite. In particular, $G$ is infinite.
\end{thm}

\subsection{Dehn monsters}

Let $G$ be a group given by a presentation $\langle X \mid R\rangle$ with a finite set of generators $X = \{x_1, . . . , x_d\}$. We refer to the canonical epimorphism $\eta: F(X) \to G$ as to the {\em projection}.

\begin{lemm}\label{af}
The following properties are equivalent for every group $G$ with a finite generating set $X$.
\begin{enumerate}
\item[(a)] For every infinite computably enumerable subset $W \subseteq F(X)$, there exist infinitely many pairs of distinct words $u_i, v_i$ such that $\eta(u_i) = \eta(v_i)$.
\item[(b)] For every infinite computably enumerable subset $W \subseteq F(X)$, there exist at least two distinct words $u, v\in W$ such that $\eta(u) = \eta(v)$.
\item[(c)] If  EP is decidable on a computably enumerable subset $W\subseteq F(X)$, then $\eta (W)$ is finite.
\end{enumerate}
\end{lemm}

\begin{proof}
Obviously (a) implies (b). Further suppose that EP is decidable on some $W\subseteq F(X)$ such that $\eta (W)$ is infinite. Then mixing the algorithm computing $W$ and the algorithm solving  EP on $W$ in the obvious way, we can compute a subset $W^\prime \subseteq W$ such that $\eta (W^\prime )$ is infinite and consists of pairwise distinct elements. This contradicts (b). Thus (b) implies (c). Notice, that (c) implies (b). Indeed if (b) does not hold, i.e., there exists an infinite computably enumerable subset $W \subseteq F(X)$ such that $\eta\vert_W$ is bijective, then  EP is obviously decidable on $W$ - different  words define different elements. Assume (b) now. If $W$ is an infinite computably enumerable set then there are words   $u_1, v_1 \in W$ such that $u_1 \neq v_1$ and $u_1 = v_1$ in $G$. Put $W^{(1)} = W\smallsetminus \{u_1,v_1\}$. Clearly $W^{(1)}$ is again an infinite computably enumerable subset of $F(X)$, so there are $u_2, v_2 \in W^{(1)}$ such that $u_2 \neq v_2$ and $u_2 = v_2$ in $G$.  Observe, that $(u_1,v_1) \neq (u_2,v_2)$. Continue this way one can show by induction  that there exist infinitely many distinct pairs of distinct words $u_i, v_i$ in  $W$ such that $\eta(u_i) = \eta(v_i)$, so (a) holds.\end{proof}

\begin{defi}\label{defaf}
A is called {\em algorithmically finite} if it satisfies one of the properties (a)--(c) from Lemma \ref{af}. A {\it Dehn monster} is an infinite recursively presented algorithmically finite group. \end{defi}

The following result shows that Definition \ref{defaf} does not depend on the generating set.

\begin{lemm}\label{changegs}
If  a finitely generated group group $G$ is algorithmically finite with respect to some finite generating set $X$ then it is algorithmically finite with respect to any finite generating set of $G$.
\end{lemm}
\begin{proof}
Let $X$ and $Y$ be two finite generating sets of $G$. Then for every $y \in Y \cup Y^{-1}$ there is a word $u_y \in F(X)$ such that $\eta(y) = \eta(u_y)$. Now if $W$ is an infinite computably enumerable subset of $F(Y)$ then given a  word $v \in W$ one can replace each letter $y \in Y\cup Y^{-1}$ by the word  $u_y$ and  get a new word $v'$, and the set $W' = \{v' \mid v \in W\} \subseteq F(X)$. The set $W'$ is an infinite computably enumerable set of words in $F(X)$.  Since $v$ and $v'$ define the same elements in $G$ it follows that two elements $v_1, v_2 \in W$ define the same element in $G$ if and only if $v_1'$ and $v_2'$ define the same element in $G$. Now the result follows.
\end{proof}

In our construction of  algorithmically finite groups we use an analogue of the notion of a simple  set of natural numbers from recursion theory (see, for example, \cite{Rogers,Cooper}).
\begin{defi}\label{simple}
A subset  $R \subseteq F(X) $ is called a {\it simple set of relations} if the following hold:
\begin{enumerate}
\item[(a)]  $R$ is computably enumerable.
\item[(b)]  The group $\langle X \mid R\rangle$ is infinite.
\item[(c)]  For every infinite computably  enumerable set $B \subseteq F(X)$ there are two distinct words $u, v \in B$ such that $u v ^{-1}\in R$.
\end{enumerate}
\end{defi}
This notion is a modification of the notion of the standard simple set of numbers (or words) in computability theory, where the condition ``$R$ is co-infinite" is replaced  by a much stronger condition (b).

\begin{lemm}\label{le:simpe-relations}
For a group presentation $G = \langle X \mid R\rangle$, the following conditions are equivalent.
\begin{itemize}
\item[(a)] $R$ is a simple set of relations.
\item[(b)] The normal closure $\langle \langle R \rangle \rangle$ of $R$ in $F(X)$ is a simple set of relations.
\item[(b)] $G$ is a Dehn Monster.
\end{itemize}
\end{lemm}
 \begin{proof}
 Directly from definitions.
 \end{proof}

  It is not hard to construct  standard simple sets  of words in $F(X)$ (see, for example, \cite{Rogers,Cooper}), but it is much harder to construct simple sets of relations.   Below we construct such a set of relations  $R$ which satisfies the Golod-Shafarevich condition. It looks a bit counter intuitive since sets satisfying the Golod-Shafarevich condition are supposed to be "sparse", and simple sets are usually viewed as "large".

\begin{thm} \label{th:simple}
For every Golod-Shafarevich group $\langle X \mid S\rangle$  there exists a  simple set of relations $R \subseteq F(X)$ such that the quotient $\langle X \mid S \cup R\rangle$ is again Golod-Shafarevich.
\end{thm}
\begin{proof}
  Let $\langle X \mid S\rangle$ be a Golod-Shafarevich presentation and $t_0 \in (0,1)$  such that  $H_p(X,S,t_0) < 0$. Given a number  $\varepsilon = |H_p(X,S,t_0)|$  we construct a simple set of relations $R$   using  a ``forcing-like" argument.

To explain our construction we need  to introduce some notation and elementary  facts from computability theory. Details can be found in any standard book on computability theory, for example, in \cite{Rogers,Cooper}.

Let $u_0, u_1,u_2, u_3, \ldots, $ be an effective bijective enumeration of all elements in $F(X)$ which preserves the length, i.e., $|u_i| \leq |u_j|$ if $i \leq j$. For example, we start with the empty word and then enumerate all words of length 1, then all words of length 2, etc. This allows us, if needed,  to identify $u_i$ with its index $i$, so computably enumerable subsets of $F(X)$ are precisely the computably subsets of $\mathbb{N}$. Furthermore, let $\tau: \mathbb{N}  \times \mathbb{N}   \to \mathbb{N}$ be a fixed computable bijection (a paring function), for example $$\tau(x,y) = \frac{1}{2}(x^2 +2xy +y^3 +3x+y)$$ (see \cite{Rogers}),   and $\pi_1, \pi_2 :\mathbb{N} \to \mathbb{N}$ the computable functions that yield the inverse mapping $\tau^{-1}$, i.e., $\tau(\pi_1(z),\pi_2(z)) = z$ for every $z \in \mathbb{N}$. The pairing function $\tau$ allows one to identify computably enumerable sets of pairs $(i,j) \in \mathbb{N} \times \mathbb{N}$ with computably enumerable sets of their images $\tau(i,j) \in \mathbb{N}$.

Let  $P_1, P_2, P_3 \ldots, $ be an effective enumeration of all Turing machines (finite programs) in the alphabet $\{0,1\}$. Denote by $U(x,y)$  a universal computable function in two variables such that for every $e \in \mathbb{N}$ the function $\phi_e(x) = U(x,e)$ is a partial computable function $\phi_e : \mathbb{N} \to \mathbb{N}$ computable by the Turing machine $P_e$.  If $W_e = dom(\phi_e )$, then $W_0,W_1, W_2, \ldots$ is a list of all computably enumerable subsets of $\mathbb{N}$, as well as of $F(X)$ and $F(X) \times F(X)$ (via the corresponding identifications). In fact, every computably enumerable set occurs infinitely many times in this list.

For $e,s,x \in \mathbb{N}$ define a partial function  $\phi_{e,s}(x): \mathbb{N} \to \mathbb{N}$ such that $\phi_{e,s}(x) = y$ if and only if the following two conditions are satisfied:
\begin{itemize}
\item $ e, x,y < s$
\item $ \phi_{e}(x) = y$   in  less  then $s$  steps  of   computation  by $P_e$ on $x$, i.e., $P_e$ starts on the input $x$ and then halts and outputs $y$ in less than $s$ steps of computation.
 \end{itemize}
 The function  $\phi_{e,s}$, as well as its domain $W_{e,s} = dom(\phi_{e,s})$,  is computable, i.e., given $x,e,s$ one can effectively verify if $\phi_{e,s}(x)$ is defined or not, and if it is defined then compute the value $\phi_{e,s}(x)$. Indeed, given $x,e,s$ one writes down the Turing machine program $P_e$, starts the computation of $P_e$ on $x$ and waits for at most $s$ steps to see if the computation halts or not. If it halts, the function $\phi_{e,s}(x)$ is defined and its value is written on the tape. Otherwise, $\phi_{e,s}(x)$ is not defined.

For every set $W_e$ put
$$
W_e^\prime = \{(u,v) \mid u,v \in W_e, u \neq v \ in \ F(X)\}.
$$
Obviously, $W_e^\prime$ is also computably enumerable.

We construct a set of relations  $R \subseteq F(X) $ by computably enumerating its elements in stages $0, 1, 2,  \ldots, s, \ldots, $ in such a way that the following requirements are satisfied for each $e \in \mathbb{N}$ (below $k(\varepsilon)$ is a fixed natural number which depends on $\varepsilon$, it will be specified  a bit later):
\begin{enumerate}
\item[({\bf L}$_e$)] $|\{r \in R \mid deg(r) \leq e \}| \leq \max \{e-k(\varepsilon),0\}$.

\item[({\bf M}$_e$)] If  $W_e$ is infinite, then $W_e^\prime \cap R \neq \emptyset.$
\end{enumerate}
Observe, that if all the conditions $({\bf L}_e)$, $e \in \mathbb{N},$ are satisfied then the following holds for the presentation $\langle X \mid S \cup R\rangle$:
$$
\begin{array}{rcl}
H(X,S \cup R,t_0) & = & H(X,S \cup R,t_0)   + \sum\limits_{i = 1}^\infty n_i(R)t_0^i \\&&\\ &\le &  H(X,S,t_0) + \sum\limits_{i = k(\varepsilon)}^\infty (i- k(\varepsilon))t_0^i \\&& \\ &=& H(X,S,t_0) + t_0^{k(\varepsilon) +1}(1+2t_0 +3t_0^2 + 4t_0^3 \ldots) \\&&\\ &=&
H(X,S,t_0) + \frac{t_0^{k(\varepsilon)+1}}{(1-t_0)^2}.
\end{array}
$$
Hence, if $k(\varepsilon)$ is chosen to satisfy
$$
\frac{t_0^{k(\varepsilon)+1}}{(1-t_0)^2} < |H_p(X,S,t_0)|,
$$
then $H(X,S \cup R,t_0) <0$ and the presentation $\langle X \mid S \cup R\rangle$
is Golod-Shafarevich. In particular,  the group $G = \langle X \mid S \cup R\rangle$ is infinite. Now we fix an arbitrary natural number $k(\varepsilon)$ satisfying the condition above. The conditions ({\bf M}$_e$), $e \in \mathbb{N}$,  ensure that $R$ (hence $S \cup R$) satisfies the condition (c) from Definition \ref{simple}.

Now we describe an algorithm $A$ that enumerates a subset $R \subseteq F(X)$ which satisfies all the conditions ({\bf L}$_e$) and ({\bf M}$_e$).

\begin{itemize}
\item At  each stage $s$ for each as yet unsatisfied condition ({\bf M}$_e$), $e < s$,  $A$ looks  for two  numbers $x_1, x_2  \in W_{e,s}$ such that the  corresponding words $ u_{x_1}, u_{x_2}\in F(X)$ are distinct,  but their images are  equal in the quotient group $F/D_{e+k(\varepsilon)}F$.  Since the Magnus embedding is effective, the quotient groups group  $F/D_{i}F$ have uniformly solvable Equality Problem. Thus $A$ can effectively either  find such a pair  $x_1, x_2$ or conclude that there is no such a pair at this stage for a given $e < s$.

\medskip

  \item If such a pair $x_1, x_2$ for an unsatisfiable $M_e$, $e < s$ occurs at the stage $s$, the algorithm includes the pair $u_{x_1}u_{x_2}^{-1} $  into $R$, and marks ({\bf M}$_e$) as now satisfied.

\medskip

  \item When the stage $s$ is finished (note that there are at most $s$ unsatisfied conditions ({\bf M}$_e$) with $e <s$ to check at this stage) the algorithm goes to the stage $s+1$.
  \end{itemize}

Observe that if $W_e$ is infinite then there are some pairs $ u_{x_1}, u_{x_2}$ of distinct words in $W_e$ that define the same element in the finite quotient $F/D_{e+k(\varepsilon)}F$. Hence, for one of such pairs  the word $u_{x_1}u_{x_2}^{-1} $ would be enumerated into $R$ at some stage $s$. Thus the set $R \subseteq F(X)$ produced by $A$ satisfies all conditions ({\bf M}$_e$), $e \in \mathbb{N}$.

To see that all the conditions ({\bf L}$_e$) are satisfied observe that if some relation  $u_{x_1}u_{x_2}^{-1}$ was added to $R$ at some stage for a given set $W_e$ then $u_{x_1} = u_{x_2}$ in $F/D_{e+k(\varepsilon)}F$, so $deg(u_{x_1} u_{x_2}^{-1}) \geq e+k(\varepsilon)$. This shows that if $uv^{-1} \in R$ and $deg(uv^{-1}) = i$ then $u=v \in W_{j}$ for some $j \leq i-k(\varepsilon)$, so
$|\{r \in R \mid deg(r) \leq i \}| \leq \max \{i-k(\varepsilon),0\},$
as claimed.

It follows  that the presentation   $G = \langle X \mid S \cup R\rangle$ is Golod-Shafarevich and the set of relations  $R$ is a simple computably enumerable set of relations, as required.
\end{proof}

Theorem \ref{main} is a simplification of the following.

\begin{coro}\label{nonamen}
There exists a recursively presented non-amenable algorithmically finite group.
\end{coro}

\begin{proof}
Let $\langle X \mid S\rangle$ be a Golod-Shafarevich presentation  with  a computably enumerable set of relations $S$. For instance, one can take $S = \emptyset$.  By Theorem \ref{th:simple}  there exists a simple set of relations $R$ such that $\langle X \mid S \cup R\rangle$ is a Golod-Shafarevich presentation. By Lemma \ref{le:simpe-relations} the group $G = \langle X \mid S \cup R\rangle$ is algorithmically finite. Observe,  that all Golod-Shafarevich groups are non-amenable \cite{EJ}. Hence the result.
\end{proof}

\begin{coro}
If $G = \langle X \mid S\rangle$ is a recursively presented Golod-Shafarevich group satisfying some group-theoretic property $Q$ which is preserved under quotients then there is a recursively  presented Golod-Shafarevich quotient of $G$ which is a  Dehn monster  satisfying $Q$.
\end{coro}

\begin{coro} \label{co:quotients}
There are recursively presented Golod-Shafarevich groups satisfying Kazhdan property $(T)$ with simple sets of defining relations.
\end{coro}
\begin{proof}
It is known that there are finitely presented Golod-Shafarevich groups with the property $(T)$  (see \cite{Ersh}, \cite{EJ}). Furthermore, the property $T$ is preserved under taking quotients. Hence the result.
\end{proof}

\section{Algorithmic and algebraic properties of Dehn monsters}

\subsection{Subgroups and quotients of algorithmically finite groups}
Let us list some elementary properties of algorithmically finite groups. Recall that a {\it section} of a group $G$ is a quotient group of a subgroup of $G$.

\begin{prop}\label{properties}
Let $G$ be an algorithmically finite group. Then the following hold.
\begin{enumerate}
\item[(a)] Every finitely generated section of $G$ is algorithmically finite.
\item[(b)] EP is undecidable in every finitely generated infinite section of $G$.
\item[(c)] $G$ is a torsion group.
\end{enumerate}
\end{prop}

\begin{proof}
Let  $G = \langle X \mid R\rangle$ be algorithmically finite. Let $H$ be a subgroup of $G$ generated by $Y=\{ y_1, \ldots , y_k\}\subseteq G$. If $H$ maps onto a group $Q$ that is not algorithmically finite, then there exists a recursive subset $W\subseteq F(Y)$ such that all elements of $W$ are pairwise distinct in $Q$. Hence all elements of $W$ are pairwise distinct in $H$. For every $y_i\in Y$ we fix a word $u_i\in F(X)$ such that $y_i=u_i$ in $G$. Replacing every $y_i$ with $u_i$ in all words from $W$, we obtain a computably enumerable subset $W^{\prime } \subseteq F(X)$ such that all elements of $W^\prime $ are pairwise distinct in $G$. Property (b) follows from (a) and Lemma \ref{af}.  Finally (c) follows from (b). Indeed since the word problem in cyclic groups is decidable, there are no infinite cyclic subgroups in $G$.
\end{proof}

\begin{coro}
There exist infinite residually finite algorithmically finite groups.
\end{coro}

\begin{proof}
Let $G$ be an algorithmically finite Golod-Sjhafarevich group. Then its image in the pro-$p$ completion is infinite by Theorem \ref{GSinf}, residually finite, and algorithmically finite by part (a) of the above proposition.
\end{proof}

It was proved by Maltsev that the word problem is decidable in every residually finite finitely presented group (see \cite{Miller}). Thus the groups constructed in the corollary above are necessarily infinitely presented. We do not know if they can be made recursively presented.

\begin{prob}
Does there exist a residually finite Dehn Monster?
\end{prob}

We also observe that every elementary amenable algorithmically finite group is finite by part (c) of Proposition \ref{properties} and the main result of \cite{Cho}. This motivates another problem.

\begin{prob}
Does there exist an infinite amenable algorithmically finite group (or, better, an amenable Dehn Monster)?
\end{prob}

\subsection{What is decidable in every recursively presented group?}

Let $\langle X \mid R\rangle$
be a presentation with finite set of generators.  Denote by $N$ the normal closure of $R$ in $F(X)$.  For elements $u_1, \ldots, u_k \in F(X)$ by $Cos(u_1, \ldots, u_k)$ we denote the union of cosets $u_1N \cup \ldots \cup u_kN$.

\begin{lemm}
Let $G = \langle X \mid R\rangle$
be a finitely generated recursively presented group.Then for any $u_1, \ldots, u_k \in F(X)$ the following holds:
\begin{enumerate}
\item[(a)] The set $Cos(u_1, \ldots, u_k)$  is an infinite computably enumerable subset of $F(X)$;
\item[(b)] The Equality Problem is decidable in $Cos(u_1, \ldots, u_k)$.
\item[(c)] The Word Problem is decidable in $Cos(u_1, \ldots, u_k)$.
\end{enumerate}
\end{lemm}
\begin{proof}
Notice that $N$, the normal closure of $R$,  is computably enumerable since $R$ is computably enumerable.  Hence,  $Cos(u_1, \ldots, u_k)$  as a finite union of cosets of $N$ is also an infinite computably enumerable subset of $F(X)$.
To see that the Word and Equality Problems are   decidable  in $Cos(u_1, \ldots, u_k)$  for each element $\eta(u_i) \in G$ fix a representative $\bar{u}_i \in F(X)$ such that $\eta(u_i) = \eta(\bar{u}_i)$ and if $\eta(u_i) = 1$ then $\bar{u}_i = \emptyset$.  Now,  given a word $w \in Cos(u_1, \ldots, u_k)$  compute the reduced forms of  the  words $bar{u}_1w^{-1}, \ldots bar{u}_kw^{-1}$ and start enumeration procedure for $N$ until one of the reduced words above appear in the enumeration process. The procedure will eventually stop and find  an element $bar{u}_k$ such that $w = bar{u}_k$ in $G$. Decidability of the Word and Equality Problems in $Cos(u_1, \ldots, u_k)$ follows.
\end{proof}

\begin{rema}
Of course, the decision algorithm above for $Cos(u_1, \ldots, u_k)$ is not uniform in $u_1, \ldots, u_k$ since we do not know how to choose the required representatives for arbitrary given words $u_1, \ldots, u_k$. So the algorithm depends on the given $u_1, \ldots, u_k$.
\end{rema}

\subsection{What is decidable in a Dehn monster?}

In this section we show that in a Dehn monster  EP is decidable only on the trivial sets of inputs.

\begin{thm} \label{th:Equality}
Let $G = \langle X \mid R\rangle$ be an infinite finitely generated recursively presented algorithmically finite group. Then the following hold:

\begin{itemize}
\item [(a)] Every computably enumerable subset $W \subseteq F(X)$ with decidable Equality  Problem in $G$ is contained in the set $Cos(u_1, \ldots,u_k)$ for some $u_1, \ldots,u_k \in F$.
\item [(b)] Every computably enumerable subset $W \subseteq F(X)$ with decidable Equality  Problem  in $G$ is negligible, i.e., $\lim\limits_{k\to \infty}\rho_k (W)= 0$. Moreover, if $G$ is non-amenable, then $W$ is strongly negligible, i.e., there exists $t>1$ such that $\rho_n(S) = O(t^{-n})$ as $n\to \infty$.
\end{itemize}
\end{thm}

\begin{proof}
Suppose that $W \subseteq F(X)$  is a  computably enumerable subset with decidable Equality  Problem in $G$. If $W$ has an infinite subset of words which are pair-wise non-equal in $G$,  then it contains a computably enumerable  infinite subset of words which are pair-wise non-equal in $G$.
Indeed, one can enumerate one-by-one all the elements $\{w_1, w_2, \ldots \}$ in $W$ and using the decision  algorithm for the Equality Problem in $W$ one can remove each word  $w_i$ if it is equal in $G$ to one of the previously listed elements $w_1, \ldots,w_{i-1}$. This yields  a countably enumerable set of words which are pair-wise non-equal in $G$, - contradiction with virtual invisibility. Hence, $W \subseteq Cos(u_1, \ldots, u_k)$ for some $u_1, \ldots, u_k \in F(X)$, as claimed.

To show 2) it suffices to prove that every coset $Cos(u) = uN$ is negligible in $F(X)$. Since the group $G$ is infinite the subgroup $N$ is of infinite index, hence negligible \cite{Woess}, i.e.,
$$\rho(N) = \lim _{k\to \infty} \rho_k(N)  = 0,$$
where $\rho_k(N) = \frac{|S_k \cap N|}{|S_k |}$ and $S_k = \{w \in F(X) \mid |w| =  k\}$. If $|u| = m$ then, obviously,
$$\rho_k(uN) \leq \rho_{k-m}(N) + \rho_{k-m +1}(N) + \ldots +\rho_{k+m}(N),$$
 so $\lim _{k\to \infty} \rho_k(uN)  = 0$.

Recall that Grigorchuk's criterion of amenability \cite{Gri} claims that $G=F(X)/N$ is  amenable  if and only if $\lim\sup |N \cap S_k|^{1/k}= 2m-1$, where $S_k$ is the sphere of radius $k$ in $F(X)$, and $m=|X|$ (see also \cite{Coh}). This  easily implies that $N$ is exponentially negligible (see \cite{BMR} for details). Thus if $G$ is non-amenable, then the argument above shows that $Cos(u_1, \ldots,u_k)$ is exponentially negligible.
\end{proof}

\begin{rema}
A similar results concerning the Conjugacy Problem (CP) also holds in Dehn monsters. Recall that the CP for a group $G$ generated by a finite set $X$ asks whether two given words from $F(X)$ represent conjugate elements of $G$.

More precisely, for $u_1, \ldots, u_k \in F(X)$, let us denote by  $Con(u_1, \ldots, u_k)$  the set of words in $F(X)$ that define elements in $G$ which are conjugate to one of the elements defined by the words $u_1, \ldots, u_k$. In particular, $Con(u)$ (where $u \in F(X)$) is the preimage in $F(X)$ of the conjugacy class of the element $\eta(u) \in G$ under the canonical projection $\eta: F(X) \to G$. It is easy to show that for every recursively presented group $G$ and any $u_1, \ldots, u_k \in F(X)$, the set $Con(u_1, \ldots, u_k)$  is an infinite computably enumerable subset of $F(X)$ and the Conjugacy Problem for $G$ is decidable in $Con(u_1, \ldots, u_k)$. If  $G$ is a Dehn monster, then every computably enumerable subset $W \subseteq F(X)$ on which the CP is decidable  is contained in the set $Con(u_1, \ldots,u_k)$ for some $u_1, \ldots,u_k \in F$.
\end{rema}

We conclude with a problem motivated by the remark above as well as by the existence of groups with decidable WP and undecidable CP \cite{Miller}.

\begin{prob}
Does there exist a group $G$ generated by a finite set $X$ such that the WP is decidable in $G$ but for every subset $S\subseteq F(X)$ with decidable CP in $G$, the image of $S$ in $G$ is contained in finitely many conjugacy classes?
\end{prob}

\medskip
{\it Acknowledgments:} we are grateful to DIMACS and participants of the workshop on "Exotic constructions in group theory" for many fruitful discussions on the subject. We also thank The Centre de recherches mathŽmatiques (Montreal) for hosting the Thematic Semester on Geometric, Combinatorial and Computational Group Theory in the Fall of 2010 which gave us an excellent opportunity to work together on this paper. Our special thanks go to Robert Gilman, Paul Schupp, and Bakh Khoussainov who's been instrumental in shaping this topic.

\end{document}